\documentclass[a4paper,11pt,reqno]{amsart}
\usepackage[T1]{fontenc}
\usepackage[utf8]{inputenc}
\usepackage{lmodern}
\usepackage{amsmath, amsthm, amssymb,amscd, mathrsfs, amsfonts, mathtools}
\usepackage{hyperref}
\usepackage{euler}
\usepackage{times}
\usepackage[all]{xy}
\usepackage{todonotes}
\usepackage{xcolor}

\usepackage[lite]{amsrefs}

\renewcommand{\PrintDOI}[1]{\href{http://dx.doi.org/\detokenize{#1}}{doi: \detokenize{#1}}%
  \IfEmptyBibField{pages}{, (to appear in print)}{}}

\def\commutatif{\ar@{}[rd]|{\circlearrowleft}}

\newcommand{\eq}[1][r]
   {\ar@<-3pt>@{-}[#1]
    \ar@<-1pt>@{}[#1]|<{}="gauche"
    \ar@<+0pt>@{}[#1]|-{}="milieu"
    \ar@<+1pt>@{}[#1]|>{}="droite"
    \ar@/^2pt/@{-}"gauche";"milieu"
    \ar@/_2pt/@{-}"milieu";"droite"}

\def\dar[#1]{\ar@<2pt>[#1]\ar@<-2pt>[#1]}
  \entrymodifiers={!!<0pt,0.7ex>+} 

\newcommand{\bigon}[4][r]{
    \ar@/^1pc/[#1]^{#2}_*=<0.3pt>{}="HAUT"
    \ar@/_1pc/[#1]_{#3}^*=<0.3pt>{}="BAS"
    \ar@{=>} "HAUT";"BAS" ^{#4}
  }

\newcommand{\bigons}[6][r]{  
    \ar@/^2pc/[#1]^{#2}_*=<0.3pt>{}="HAUT"
    \ar@{}    [#1]     ^*=<0.3pt>{}="MILIEUHAUT"
                       _*=<0.3pt>{}="MILIEUBAS"
    \ar[#1]_(0.3){#3}                  
    \ar@/_2pc/[#1]_{#4}^*=<0.3pt>{}="BAS"
    \ar@{=>} "HAUT";"MILIEUHAUT" ^{#5}
    \ar@{=>} "MILIEUBAS";"BAS" ^{#6}
  }

\newtheorem{thm}{Theorem}[section]
\newtheorem{pro}[thm]{Proposition}
\newtheorem{lem}[thm]{Lemma}
\newtheorem{cor}[thm]{Corollary}

\theoremstyle{definition}
\newtheorem{df}[thm]{Definition}

\theoremstyle{remark}
\newtheorem{rmk}[thm]{Remark}

\newtheorem{ex}[thm]{Example}
\newtheorem{exs}[thm]{Examples}
\allowdisplaybreaks

\newcommand{\Z}{\mathbb{Z}}

\newcommand\rTo{\longrightarrow}

\newcommand\rRack{\triangleleft}

\def\t{\theta}

\def\utr{\, \underline{\triangleright}\, }
\def\otr{\, \overline{\triangleright}\, }

\hypersetup{
  colorlinks = true,
  urlcolor = blue,
  linkcolor = blue,
  citecolor = red,
  pdfauthor = {Elhamdadi, M., Liu, M., Nelson, S.} 
  pdfkeywords = {Quandles, quasi-trivial quandles and biquandles, cocycle invariants },
  pdftitle = {Elhamdadi, M., Liu, M., Nelson, S. - Quasi-trivial Quandles and Biquandles,  Cocycle Enhancements and Link-Homotopy of Pretzel links},
  pdfsubject = {quandles},
  pdfpagemode = UseNone
}

\title[QT-(Bi)Quandles, Cocycle Enhancements and Link-Hom. of Pretzel Links]{Quasi-trivial Quandles and Biquandles,  Cocycle Enhancements and Link-Homotopy of Pretzel links  }

\author{Mohamed Elhamdadi} 
\address{Department of Mathematics, 
University of South Florida, Tampa, FL 33620, U.S.A.} 
\email{emohamed@math.usf.edu} 

\author{Minghui Liu} 
\address{Florida College, Temple Terrace, FL 33617, U.S.A.} 
\email{LiuM@floridacollege.edu} 

\author{Sam Nelson} 
\address{Claremont McKenna College, Claremont, CA 91711, U.S.A.} 
\email{Sam.Nelson@claremontmckenna.edu}

\begin{document}

\maketitle

\begin{abstract}
We investigate some algebraic structures called {\it quasi-trivial} quandles and we use them to study link-homotopy of pretzel links. Precisely, a necessary and sufficient condition for a pretzel link with at least two components being trivial under link-homotopy is given. We also generalize the quasi-trivial quandle idea to the case of biquandles and consider enhancement of the quasi-trivial biquandle cocycle counting invariant by quasi-trivial biquandle cocycles, obtaining invariants of link-homotopy type of links analogous to the quasi-trivial quandle cocycle invariants in Inoue's article \cite{I}.
   
\end{abstract}

\tableofcontents

\section{Introduction}
In 1954, Milnor \cite{Milnor} introduced a concept called link-homotopy, which significantly simplified the theory of classical links.  Two links are link-homotopy equivalent if one can be transformed to the other by a finite sequence of ambient isotopies where no crossing change is allowed between distinct components of the link but crossing changes are allowed on the same component. Also, in \cite{Milnor} he classified link homotopy for links with up to three components and gave criteria determining when a link is trivial under link-homotopy. The case of four-components link was settled in \cite{Levine} by Levine. In \cite{HL} Habegger and Lin gave a complete classification of links of arbitrarily many components up to link homotopy. 

In this article we study link-homotopy using {\it quasi-trivial quandles}, 
i.e., quandles whose orbits subquandles (see \cite{NW} for instance) are trivial.
Precisely, we give a necessary and sufficient condition for a pretzel link with at least two components to be trivial under link-homotopy. We also generalize the quasi-trivial quandle idea to the case of biquandles and consider enhancement of the quasi-trivial biquandle cocycle counting invariant by quasi-trivial biquandle cocycles, obtaining invariants of link-homotopy type of links analogous to the quasi-trivial quandle cocycle invariants in Inoue's article \cite{I}.

The paper is organized as follows.  A brief review of quandles is given in Section~\ref{prelimsec}.  In Section~\ref{QuasiQ}, we recall the definition of quasi-trivial quandles and give examples. Coloring of links by quasi-trivial quandles is studied in Section~\ref{Color}.  A necessary and sufficient condition for a pretzel link with at least two components to be trivial under link-homotopy is given in Section~\ref{HPL}.  In Section ~\ref{QuasiBiq}, we generalize the quasi-trivial quandle idea to the case of 
biquandles and consider enhancement of the quasi-trivial biquandle cocycle 
counting invariant by quasi-trivial biquandle cocycles, obtaining invariants 
of link-homotopy type of links analogous to the quasi-trivial quandle
cocycle invariants in \cite{I}.  The last section contains some open problems for future research.

\section{Basics of Quandles}\label{prelimsec}
In this section, we collect some basics about quandles that we will need through the paper.  We begin with the following definition taken from \cites{EN, Joyce, Matveev}

\begin{df} \label{quandledef}
	A {\it quandle} is a set $X$ provided with a binary operation 

$\rRack:  X\times X  \rTo X$ which maps $(x,y)$ to $x\rRack y$, such that 
\begin{itemize}
\item[(i)] for all $x,y\in X$, there is a unique $z\in X$ such that $y=z\rRack x$;
\item[(ii)]({\it right distributivity}) for all $x,y,z\in X$, we have $(x\rRack y)\rRack z=(x\rRack z)\rRack (y\rRack z)$;
\item[(iii)] ({\it idempotency}) for all $x\in X$, $x\rRack x=x.$
\end{itemize}
\end{df}
If furthermore $(x\rRack y)\rRack y=x$, for all $x, y \in X$, then the quandles is called a {\it kei} (or involutive quandle).  

Observe that property (i) also reads that for any fixed element $x\in X$, the map $R_x:X\rTo X$ sending $y$ to $y\rRack x$ is a bijection. 
Also, notice that the distributivity condition is equivalent to the relation $R_x(y\rRack z)=R_x(y)\rRack R_x(z)$ for all $y,z\in X$.

We give few examples of quandles here.  More examples can be found in \cite{EN}.

\begin{exs}
\noindent
	\begin{enumerate}
		
		\item
		Any non-empty set $X$ with the operation $x \rRack y=x$ for any $x,y \in X$ is
		a quandle called the {\it trivial} quandle.	
		
		\item
		Let $n$ be a positive integer.
		For  
		$x, y \in \Z_n$ (integers modulo $n$), 
		define
		$x \rRack y = 2y-x \pmod{n}$.
		Then the operation $\rRack$ defines a quandle
		structure  called the {\it dihedral quandle} and denoted
		$R_n$.
		
		\item
	Let $G$ be a group.  The operation $x\rRack y=y^{-1}xy$   makes $G$ into a quandle which is denoted by $Conj(G)$ and is called the {\em conjugation quandle of $G$}. 
	
	\item
	
	The operation 
	$x\rRack y=yx^{-1}y$ makes any group $G$ into a kei and is called the {\em core quandle of $G$}.
	
	\item
	Any ${\Z }[t, t^{-1}]$-module $M$
	is a quandle with
	$x \rRack y=tx+(1-t)y$, $x,y \in M$, called an {\it  Alexander  quandle}.
	
	\end{enumerate}
\end{exs}

A quandle homomorphism $\phi$ from $(X,\ast)$ to $(Y,\rRack)$ is a map from $X$ to $Y$ satisfying $\phi(x\ast y)=\phi(x)\rRack\phi(y)$ for all $x,y \in X$. A quandle homomorphism $\phi: X \rightarrow  X$ that is a bijection is called a quandle {\it automorphism}.  The set of all quandle automorphisms of $X$ forms a group denoted by {\rm Aut(X)}.  The subgroup of {\rm Aut(X)} generated by all bijections $R_x$ is called the {\it inner automorphism group} of $X$ and denoted by {\rm Inn}$(X)$.  The action of the inner automorphism group {\rm Inn}$(X)$ on the quandle $X$ gives the decomposition of $X$ in term of its orbits.  
A quandle $X$ is called {\it connected} if the inner automorphism group {\rm Inn}$(X)$ acts transitively on $X$ (that is there is only one orbit).  For example, the odd dihedral quandles $R_{2n+1}$ are connected.  The orbit decomposition of the dihedral quandle $R_6$ is $\{0,2,4\} \sqcup \{1,3,5\}$.

\section{Quasi-trivial Quandles}\label{QuasiQ}
All the quandles in this paper will be non-connected quandles except if stated otherwise.  A reason for which we are interested in non-connected quandles in this article is their applications to $n$-components links $L=K_1\sqcup \cdots \sqcup K_n \subset \mathbb{S}^3$ with $n \geq 2$. Precisely, we apply quandles to study links up to link homotopy.  We define quasi-trivial quandles following the idea in \cite{H}.

\begin{df}
	A quandle in which  $x\rRack y= x,$ for all $x$ and $y$ belonging to the same orbit, is called a {\it quasi-trivial quandle}.  
\end{df}

\begin{ex}
	Let $Q$ be a quandle and let $RQ$ be the quandle obtained from $Q$ by adjoining the relation that $x\rRack y= x,$ for all $x$ and $y$ belonging to the same orbit. Then $RQ$ is a quasi-trivial quandle.  In particular if $Q(L)$ is the fundamental quandle of a link $L$ \cite{Joyce} then the quasi-trivial quandle $RQ(L)$ is shown to be a link homotopy invariant in ~\cite{H, I}.  It is called the {\it reduced fundamental quandle} in ~\cite{H}.
\end{ex}

\begin{ex}
	The dihedral quandle $R_4=\{0,2\} \sqcup \{1,3\}$ is a quasi-trivial quandle.  Note that $R_{2n}$ is not a quasi-trivial quandle for $n \geq 3$. 
\end{ex}

\begin{ex}
	Any group $G$ is a quandle under the binary operation $x\rRack y=y^{-1}xy$. Under this operation, $G$ is a quasi-trivial quandle if and only if $G$ is a $2-$Engel group; that is, $G$ satisfies the condition that $x$ commutes with $g^{-1}xg$ for all $x,g\in G$. One specific example is the quaternion group $Q_8=\langle {i},{j},{k}\mid {i}^2={j}^2={k}^2={i}\;{j}\;{k}=-1; (-1)^2=1\rangle$.  This quandle has the orbit decomposition $Q_8=\{1\}\sqcup \{-1\}\sqcup \{ \pm {i} \}\sqcup \{ \pm {j} \}\sqcup \{ \pm {k} \}$.
\end{ex}

\begin{ex}
	In an Alexander quandle, two elements $x$ and $y$ are in the same orbit if and only if there exists $z$ such that $x-y=(1-t)z\;$ (~\cite{N1,NP}).  As a consequence, the Alexander quandle $A=\mathbb{Z}_n[t^{\pm 1}]/(1-t)^2$ is a Quasi-trivial quandle for any positive integer $n$.
\end{ex}

\section{Colorings of Links by Quasi-trivial Quandles}\label{Color}

In 1954, Milnor \cite{Milnor} defined the notion of Link-homotopy to study links.  Two links are link-homotopy equivalent if one can be transformed to the other one by a finite sequence of ambient isotopies where no crossing change is allowed between distinct components of the link but crossing changes are allowed on the same component.    
 
A coloring of a link $L$ by a quandle $X$ is a quandle homomorphism from the fundamental quandle $Q(L)$ of the link $L$ to the quandle $X$ (See \cite{EN} for more details). In \cite{H}, it was shown that colorings of links by quasi-trivial quandles are invariant of homotopy links.  In \cite{I}, the author defined a (co)homology theory for quasi-trivial quandles.  He also showed that quandle cocycles associated with $2$-cocycles of quasi-trivial quandles are link homotopy invariants.

\begin{ex}
	Two $2$-torus links $(2,2m)$ and $(2,2n)$ with $m,n\geq 0$ are homotopically equivalent if and only if $m=n$. 
	To see this, we consider the torus link $(2,\left|m-n\right|)$ if necessary, it suffices to prove that if $n>0$, the torus link $T(2,2n)$ is homotopically non-trivial. Let $B_2$ be the braid group with two strings and consider the braid $\sigma_1^{2n}$ whose closure is the link $T(2,2n)$. 
	
	For all $k\geq 2$, using the Alexander quandle ${\mathbb{Z}_k[t^{\pm1}]}/(1-t)^2$, if we color the top arcs of the braid $\sigma_1^2\in B_2$ by the vector $(0,1)$, then the bottom color vector is $(1-t, t)$.  Since $t \neq 1$, we have that the braid $\sigma_1^2$ is homotopically non-trivial using coloring in ${\mathbb{Z}_k[t^{\pm1}]}/(1-t)^2$ for all $k\geq 2$. For the braid $\sigma_1^{2n}$, using coloring in ${\mathbb{Z}_{2n-2}[t^{\pm1}]}/(1-t)^2$, we can prove that $\sigma_1^{2n}$ is homotopically trivial only if $\sigma_1^2$ is homotopically trivial, which is false and thus the proof is completed. 
\end{ex}
\section{Homotopy Pretzel Links}\label{HPL}
It is known (see for example \cite{Kawauchi}) that  pretzel link $(p_1,p_2,\ldots,p_n)$ is a knot if and only if both $n$ and all the $p_i$ are odd or exactly one of the $p_i$ is even.  Since every knot is trivial in the homotopy sense, we only consider the case that the pretzel link $(p_1,p_2,\ldots,p_n)$ has at least two components. 

 In the following we will consider colorings by the Alexander quandle ${\mathbb{Z}_k[t^{\pm1}]}/(t-1)^2$ with $x \rRack y=tx+(1-t)y$.  

\[\includegraphics{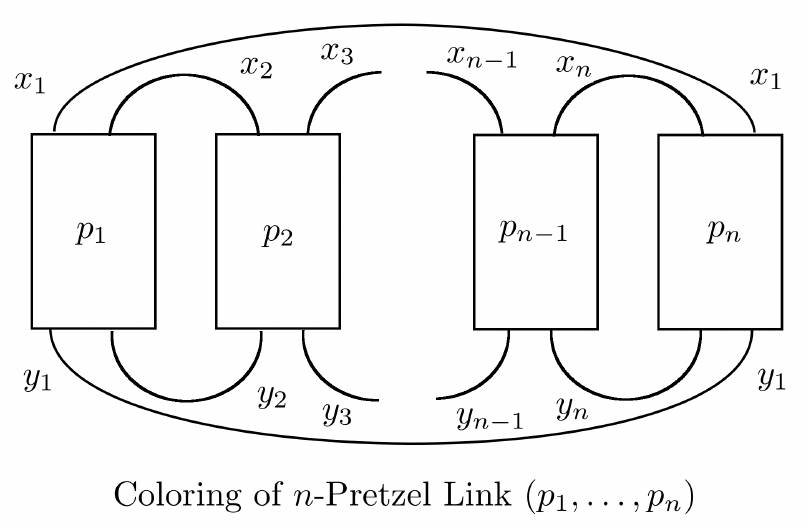}\]

  We denote the top color on the link by $(x_1, x_2, \cdots, x_n)$ and the bottom color by $(y_1, y_2, \cdots, y_n)$ as can be seen in the figure.  For fixed $p_1, \cdots, p_n$ and $k$, note that the bottom vector is uniquely determined by the top vector, therefore the number of colorings is at most $k^{2n}$.

Since we will be studying pretzel links, it will be convenient to have the powers of the matrix  $A=\begin{bmatrix}
0&1\\t&1-t
\end{bmatrix}$.  For an arbitrary integer $k\geq 2$, let 
\[B=A^2=\begin{bmatrix}
t&1-t\\-t^2+t&t^2-t+1
\end{bmatrix}=\begin{bmatrix}
t&1-t\\1-t&t
\end{bmatrix}\] in $M_{2\times 2}\left({\mathbb{Z}_k[t^{\pm1}]}/(t-1)^2\right)$ . An induction argument gives the following lemma
\begin{lem}\label{matrixpower}
	For any non-negative integer $j$,  
	\[B^j= \begin{bmatrix}
	jt-j+1&j-jt\\j-jt& jt-j+1
	\end{bmatrix}.\]
\end{lem} 
We then have the following

\begin{cor}\label{matrix}
	In $\mathbb{Z}_k[t^{\pm1}] /(t-1)^2 $, the smallest positive integer $l$ such that $B^l=I_2$ is $l=k$.
\end{cor} 

A straightforward computation gives the following lemma 
\begin{lem}\label{Triv}
	Consider the braid $\sigma_1^k$ in $B_2$.  Using coloring over the Alexander quandle ${\mathbb{Z}_k[t^{\pm1}]}/(t-1)^2$.  Let the top color and the bottom color are respectively $(x,y)$ and $(z,u)$.  Then $x=z$ if and only if $y=u$.
\end{lem}

	Since the case of $2$-pretzel link is different than the case of $n$-pretzel link with $n\geq 3$, we consider the following cases:\\
	 

\subsection*{Homotopy $2$-Pretzel Links}

By considering every homotopy $2$-Pretzel link as a torus link, we conclude that
\begin{lem}
The homotopy $2$-Pretzel link $(p_1,p_2)$ is homotopically trivial if and only if $p_1+p_2=0$.
\end{lem}

\subsection*{Homotopy $n$-Pretzel Links}
Let $(p_1,p_2,\ldots,p_n)$ be an $n$-Pretzel link with $N$ components.  Then one easily proves the following lemma:
\begin{lem}
	  
If all $p_i$s are odd, then $N=\frac{3+(-1)^n}{2}$; but if at least one of the $p_i$s is even, then $N$ is the number of even $p_i$s. 
\end{lem}

\begin{lem}\label{evenpretzel}
	Let $n \geq 3$ and let all $p_i$s be even numbers. Then the $n$-Pretzel link $(p_1,p_2,\ldots,p_n)$ is homotopically trivial if and only if all the $p_i$s are zero. 
\end{lem}
\begin{proof}
Assume that the $n$-Pretzel link $(p_1,p_2,\ldots,p_n)$ is homotopically trivial. Now fix a positive integer $k$ and color the link by the quasi-trivial quandle ${\mathbb{Z}_k[t^{\pm1}]}/(t-1)^2$. 
For each $i$-th box, since $p_i$ is even, $x_i$ and $y_i$ are in the same orbit and thus $x_i=y_i$.  The same argument gives  that $x_{i+1}=y_{i+1}$. Now by Corollary \ref{matrix}, the integer $\frac{p_i}{2}$ is a multiple of $k$. Since $k$ is arbitrary, $p_i=0$ for every $i$.
	
\end{proof}

\begin{lem}
Assume that at least two $p_i$s are even, then the pretzel link is homotopically trivial if and only if all the even $p_i$s are zero.
\end{lem}
\begin{proof} 
If all $p_i$s are zero, then the pretzel link is a disjoint union of $N$ knots, where $N=\#$ of even $p_i$s, thus it is homotopically trivial.\\
Now we prove the "only if" part.  By applying self-crossing changes if necessary, we assume that any odd $p_i$ is equal to $1$.  By applying {\it flypes} over the boxes with even $p_i$s, if necessary, we assume that the pretzel link has the form
\[ (2j_{1}, \underbrace{1, \ldots, 1}_{s}, 2j_2, \ldots, 2j_N).
\]
Fix an integer $k\geq 3$, we color the pretzel link using the quasi-trivial quandle ${\mathbb{Z}_k[t^{\pm1}]}/(t-1)^2$ as in the figure
\[\includegraphics{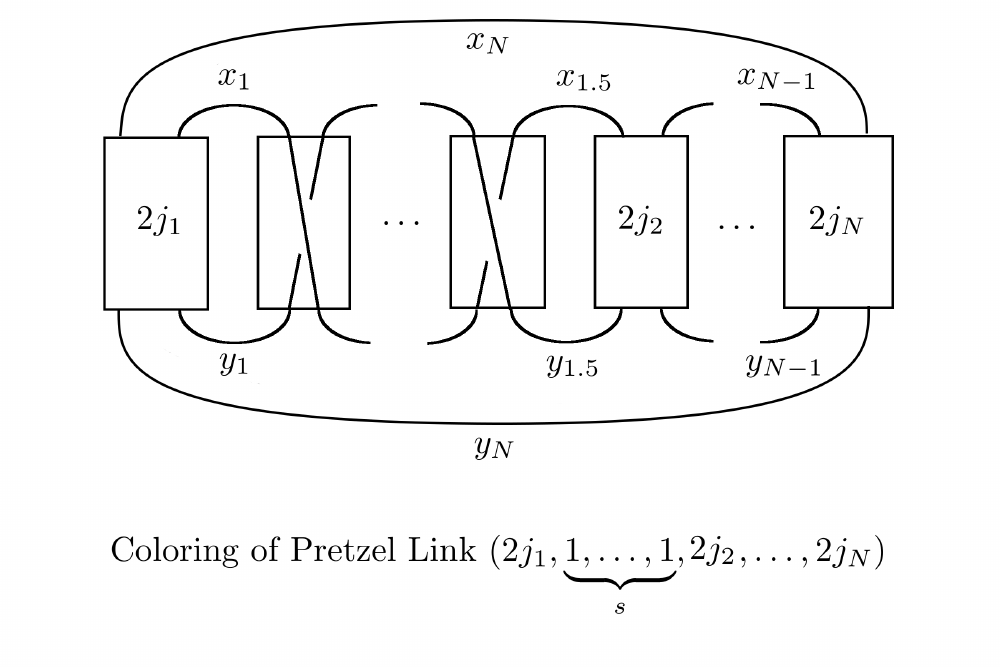}\]
Note that the coloring is completely determined by the values of $x_1, \dots, x_N$.  By choosing $x_1=1$ and $x_i=0$ for $2 \leq i \leq N$, we conclude that $j_1=j_2$.  A similar argument shows that $j_1=j_2=\cdots=j_N$.  By applying self-crossing changes, we can assume that $s=0$ or $s=1$.  The case $s=0$ is proved in Lemma~\ref{evenpretzel}.  Now assume that $s=1$, choosing $x_N=0$, $x_1=1$ and $x_2=j_1(t-1)+1$ implies that $j_1=0$ and thus this reduces also to the case proved in Lemma~\ref{evenpretzel}. 

\end{proof}

\begin{lem}
	Let $n=2m \geq 4$ be a positive integer  and let all $p_i$s be odd numbers. Then the $n$-Pretzel link $(p_1,p_2,\ldots,p_n)$ is not homotopically trivial. 
\end{lem}
\begin{proof}  In this proof, we are using the quasi-trivial quandle $\mathbb{Z}[t^{\pm 1}]/(1-t)^2$. \\ For $  i \in \{1, \cdots, n\}$, let $p_i=2k_i+1$.  Recall that 
		$A= \begin{bmatrix}
		0&1\\
		t& 1 -t
		\end{bmatrix}$ and thus 
		\[A^{2k_i+1}= \begin{bmatrix}
		k_i -k_i t&1-k_i+k_i t\\
		t-k_i+k_i t& 1-k_i t +k_i -t
		\end{bmatrix}.\]
		 Let the top color on the link be $(x_1, x_2, \cdots, x_n)$, then one obtains the following equation 
		 \begin{equation}\label{x13}
		 (k_1(t-1)+t)x_1 + (1-t)(1+k_1-k_2)x_2 + (k_2(1-t)-1)x_3=0.
		 \end{equation}
		 Multiplying both sides of this equation by $t-1$ gives $(x_3-x_1)(t-1)=0$, thus $x_3=x_1 + \lambda (t-1)$ for some constant $\lambda \in \mathbb{Z}$.  Substituting $x_3$ in the equation~\ref{x13}, we get 
		 \[
		 (k_1(t-1)+t +(k_2(1-t)-1))x_1 + (1-t)(1+k_1-k_2)x_2 -\lambda (t-1)=0,
		 \]
		 thus
		 \[
		 x_3=x_1  + (1-t)(1+k_1-k_2)(x_2-x_1).
		 \]
		By induction we obtain that for all $i$, $x_i$ is completely determined by $x_1$ and $x_2$.  Thus the coloring space of the $n$-pretzel homotopy link is at most two-dimensional over $\mathbb{Z}$.  Similarly, we obtain that for any $l \in \mathbb{Z}_n$
		
		\[
		x_l=x_{l-2}  + (1-t)(1+k_{l-2}-k_{l-1})(x_{l-1}-x_{l-2}).
		\] 
		
		Let us denote $1+k_{l-1}-k_{l}$ by $a_l$ and choose $x_1=0$ and $x_2=1$, by induction we obtain that 
		
			\[
			x_{2j-1}=(a_2 +a_4 + \cdots +a_{2j-2})(1-t),  
			\] 
			and 
			
			\[
			x_{2j}=1-(a_3 +a_5 + \cdots +a_{2j-1})(1-t).  
			\] 	
			Therefore 
			\[x_1=x_{2m+1}=(a_2 +a_4 + \cdots +a_{2m})(1-t)\]
			 and \[x_2=x_{2m+2}=1-(a_3 +a_5 + \cdots +a_{2m+1})(1-t)=1-(a_1+ a_3 + \cdots +a_{2m-1})(1-t).\]
			 Since $x_1=0$ and $x_2=1$, we obtain that 
			 \[a_1+ a_3 + \cdots +a_{2m-1}=0\] and \[a_2 +a_4 + \cdots +a_{2m}=0.\] Adding these two equations gives the contradiction that $2m=0$.  Therefore, the choice of $x_1=0$ and $x_2=1$ gives a vector that does not color the $2m$-pretzel homotopy link and thus this link is not trivial under link-homotopy.  
		 
\end{proof}


Thus we have proved the following theorem:

\begin{thm}
	The pretzel link $(p_1, \ldots, p_n)$ with at least two components is homotopically trivial if and only if one of the following two conditions is satisfied:
	\begin{itemize}
		\item 
		$n=2$ and $p_1+p_2=0$
		\item
		$n \geq 3$ with the assumption that not all $p_i$ are odd, then every $p_j$ is either odd or zero.
	\end{itemize}
\end{thm}

\noindent
Now we summarize our discussion on the triviality of pretzel links under link-homotopy as follows.  Let $E$ denote the number of even $p_i$s in a pretzel link $(p_1, \ldots, p_n)$.  We have the following:

\begin{enumerate}
	\item 
	$n=2$
	\begin{enumerate}
		\item $p_1+p_2$ is odd, the link is a the trivial knot.
		\item $p_1+p_2$ is even
		\begin{enumerate}
			\item 	$p_1+p_2=0$, the link is trivial.
			\item 
			$p_1+p_2\neq 0$, the link is non-trivial.
		\end{enumerate}

	\end{enumerate}
	\item
	$n \geq 3$
	\begin{enumerate}
		\item 
		$E=0$
		\begin{enumerate}
			\item 
			$n$ is even, the link is two-component non-trivial.
			
			\item
			$n$ is odd, the link is the trivial knot.
		\end{enumerate}
		\item
		$E=1$, the link is the trivial knot.
		
		\item
		$E\geq 2$
		\begin{enumerate}
			\item 
			every even $p_i$ is zero, the link is trivial.
			
			\item
			at least one even $p_i$ is not zero, the link is non-trivial. 
		\end{enumerate}

	\end{enumerate}
	
\end{enumerate}

Bonahon gave a classification of Montesinos links  \cite{Bonahon} (see also \cite{BurdeZieschang}), from which the following is a corollary: 

\begin{cor}\label{Bonahon}
Given a pretzel link $(p_1, \ldots, p_n)$ with $n \geq 3$ and assume that for each $p_i$, the absolute value $|p_i| >1,$ then the pretzel link is non-trivial.  Furthermore, two such pretzel links are isotopic if and only if they have the same list $(p_1, ...., p_n)$ up to cyclic permutations and reversal of order.	
\end{cor}

Let $q$ be the map from the set of pretzel links to the set of homotopy pretzel links defined in the way that $q$ maps every link to its homotopy class. Using Corollary ~\ref{Bonahon}, we give a description of the kernel of the map $q$ in the following corollary:

\begin{cor}\label{Cor}

A pretzel link $l=(p_1,\ldots , p_n)$ is in the kernel of $q$ if and only if either $l$ is a trivial link, or if $n=2$ with $p_1+p_2=0$ or odd, or $n\geq 3$ with one of the following three cases
\begin{itemize}
	\item 

The number $n$ is odd and all $p_i$s are odd,
\item
There is exactly one $p_i$ that is even,

\item
There are at least two $p_i$s being even and all the even $p_i$s are zeros.
\end{itemize}
\end{cor}
\begin{rmk}
	
In Corollary ~\ref{Cor} we include the case that the link $l$ is a knot, which is a trivial link in the homotopy sense.  In the context of braids, a description of the kernel of the map $q:B_n\rightarrow \tilde{B}_n$ is given by Goldsmith in \cite{Goldsmith}, where $B_n$ and $\tilde{B}_n$ are braid groups and the homotopy braid group with $n$ strings respectively; See also \cite{Liu} and \cite{MK}.
\end{rmk}

\section{Quasi-trivial Biquandles and Cocycle Enhancements}\label{QuasiBiq}

In this section we generalize the quasi-trivial quandle idea to the case of 
biquandles and consider enhancement of the quasi-trivial biquandle cocycle 
counting invariant by quasi-trivial biquandle cocycles, obtaining invariants 
of link-homotopy type of links which generalize the quasi-trivial quandle
cocycle invariants in \cite{I}. 

\begin{rmk}
While the knot quandle is a complete invariant up to mirror image for classical
knots, there are standard examples (e.g., the Kishino knot) of virtual knots 
whose knot quandle is trivial but whose biquandle is nontrivial. Thus, for the 
purpose of studying link-homotopy in the virtual setting it will be useful to 
extend the notion of quasi-triviality to biquandles.
\end{rmk}

\begin{df}
A \textit{biquandle} is a set $X$ with two binary operations 
$\utr,\otr:X\times X\to X$ satisfying for all $x,y,z\in X$
\begin{itemize}
\item[(i)] $x\otr x=x\utr x$,
\item[(ii)] The maps $\alpha_y,\beta_y:X\to X$ and $S:X\times X\to X\times X$
defined by $\alpha_y(x)=x\otr y$, $\beta_y(x)=x\utr y$ and 
$ S(x,y)=(y\otr x,x\utr y)$ are invertible, and
\item[(iii)] The \textit{exchange laws} are satisfied:
\[\begin{array}{rcl}
(x\utr y)\utr (z\utr y) & = & (x\utr z)\utr (y\otr x) \\
(x\otr y)\utr (z\otr y) & = & (x\utr z)\otr (y\utr x) \\
(x\otr y)\otr (z\otr y) & = & (x\otr z)\otr (y\utr x).
\end{array}\]
\end{itemize}
\end{df}

\begin{ex}
Some standard examples of biquandles include
\begin{itemize}
\item (\textit{Constant Action Biquandles}) For any set $X$ and bijection
 $\sigma:X\to X$, the operations \[x\utr y=x\otr y=\sigma (x)\] define a 
biquandle structure on $X$.
\item (\textit{Alexander Biquandles}) For any module $X$ over 
$\mathbb{Z}[t^{\pm 1},s^{\pm 1}]$, the operations
\[x\utr y=tx+(s-t)y,\quad x\otr y=sx\]
define a biquandle structure on $X$.
\item (\textit{Biquandles Defined by Operation Tables}) For a finite set 
$X=\{x_1,\dots, x_n\}$, we can specify the operation tables of a biquandle
structure with an $n\times 2n$ block matrix $M$ whose $(j,k)$ entry $x_l$ 
satisfies
\[x_l=\left\{\begin{array}{ll}
x_j\utr x_k & 1\le k\le n \\
x_j\otr x_k & n+1\le k\le 2n
\end{array}\right.\]
For instance, the Alexander biquandle $X=\mathbb{Z}_3=\{1,2,3\}$ (we use 3 
for the class of zero since our row and column numbers start at 1) with
$t=1$ and $s=2$ has matrix
\[\left[\begin{array}{rrr|rrr}
2 & 3 & 1 & 2 & 2 & 2 \\ 
3 & 1 & 2 & 1 & 1 & 1 \\
1 & 2 & 3 & 3 & 3 & 3
\end{array}\right].\]
\end{itemize}
\end{ex}

Unlike the case of quandles, the actions of $y\in X$ on $X$ given by 
$\alpha_y,\beta_y$ are not automorphisms of $X$ in general, but they are
still permutations of the elements of $X$ and generate a subgroup $CG$ of the 
symmetric group on $X$ known as the \textit{column group} in \cite{HN}.
The orbits of the action of $CG$ on $X$ partition $X$ into disjoint 
orbit sub-biquandles analogously to the quandle case.

\begin{ex}\label{ex:abq413}
The Alexander biquandle $X=\mathbb{Z}_4$ with $t=1$ and $s=3$ has matrix
\[M_X=\left[\begin{array}{rrrr|rrrr}
3 & 1 & 3 & 1 & 3 & 3 & 3 & 3 \\
4 & 2 & 4 & 2 & 2 & 2 & 2 & 2 \\
1 & 3 & 1 & 3 & 1 & 1 & 1 & 1 \\
2 & 4 & 2 & 4 & 4 & 4 & 4 & 4
\end{array}\right]\]
and thus has orbit sub-biquandles $O_1=\{1,3\}$ and $O_2=\{2,4\}$,
with matrices
\[M_{O_1}=\left[\begin{array}{rr|rr}
2 & 2 & 2 & 2 \\
1 & 1 & 1 & 1
\end{array}\right]\]
and
\[M_{O_2}=\left[\begin{array}{rr|rr}
1 & 1 & 1 & 1 \\
2 & 2 & 2 & 2 \\
\end{array}\right].\]
\end{ex}

\begin{df}\label{ex:4b}
A biquandle $X$ is \textit{quasi-trivial} if its orbit sub-biquandles are trivial,
i.e., if for all $x,y\in X$, if $x$ and $y$ are in the same orbit, then
$x\utr y= x\otr y=x$.
\end{df}

\begin{ex}\label{ex:redb}
The biquandle in Example \ref{ex:abq413} is \textit{not} a quasi-trivial biquandle 
since it has one trivial and one non-trivial orbit sub-biquandle. However,
the biquandle with operation matrix
\[M_X=\left[\begin{array}{rrrr|rrrr}
1 & 1 & 1 & 1 & 1 & 1 & 2 & 2 \\
2 & 2 & 2 & 2 & 2 & 2 & 1 & 1 \\
4 & 4 & 3 & 3 & 4 & 4 & 3 & 3 \\
3 & 3 & 4 & 4 & 3 & 3 & 4 & 4
\end{array}\right]\]
is a quasi-trivial biquandle since both of its orbit sub-biquandles $O_1=\{1,2\}$
and $O_2=\{3,4\}$ are isomorphic to the trivial biquandle on two elements.
\end{ex}

\begin{df}
Let $L$ be an oriented link diagram and $X$ a biquandle. Then a 
\textit{biquandle coloring of $L$ by $X$}, also called an
\textit{$X$-coloring} of $L$, is an assignment of elements of $X$ to the
semiarcs in $L$ (the edges in $L$ considered as a directed $4$-valent graph)
such that at each crossing we have one of the following local pictures:
\[\includegraphics{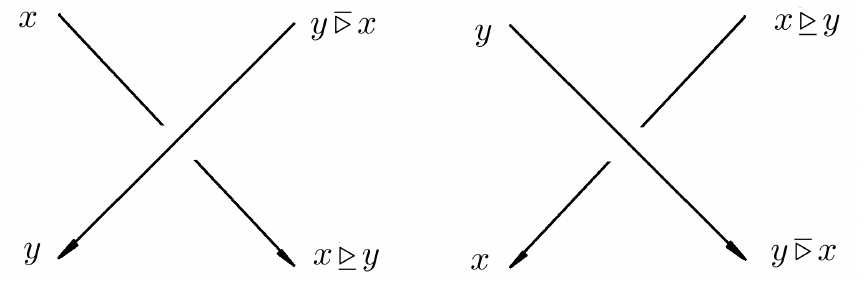}\]
\end{df}

The biquandle axioms are the conditions needed to guarantee that for any 
$X$-coloring of a link diagram before an oriented Reidemeister move, there is
a unique $X$-coloring of the resulting diagram after the move which agrees
with the original outside the neighborhood of the move; see \cite{EN} for more.
It follows that the number of $X$-colorings of an oriented link diagram is
an invariant of oriented links, which we denote by 
\[\Phi_X^{\mathbb{Z}}(L)=|\{X\mathrm{-colorings\ of}\ L\}|.\]

Exactly as in the quandle case, self-crossing changes in a link diagram do 
not change the set of colorings of a link by a quasi-trivial biquandle $X$ since 
the colors on the semiarcs at a self-crossing are from the same orbit.
Thus, we have
\begin{pro}
Let $L,L'$ be oriented links and $X$ a quasi-trivial biquandle. If $L$ and $L'$
are link homotopic, then $\Phi_X^{\mathbb{Z}}(L)=\Phi_X^{\mathbb{Z}}(L').$
\end{pro}

\begin{ex}
The Hopf link and the $(4,2)$-torus link have 8 and 16 colorings
respectively by the quasi-trivial biquandle $X$ in Example \ref{ex:redb}, so the
biquandle counting invariant $\phi_X^{\mathbb{Z}}$ detects that these two links 
are not link-homotopic.
\end{ex}

Now, let $R$ be a commutative ring, let $C_n(X;R)=R[X^n]$, be the free 
$R$-module 
on ordered $n$-tuples of 
elements of $X$ and let $C^n(x;R)=\mathrm{Hom}(C_n(X;R),R)$. Then the maps
\[\partial_n:C_n(X;R)\to C_{n-1}(X;R)\quad \mathrm{and} \quad
\delta^n:C^{n}(X;R)\to C^{n+1}(X;R)
\]
defined by
\[\partial_n(\vec{x})=\sum_{k=1}^n(-1)^k[\partial_n^{0,k}(\vec{x})-\partial_n^{1,k}(\vec{x})]\]
where 
\begin{eqnarray*}
\partial_n^{0,k}(x_1,\dots, x_n) & = & (x_1,\dots, x_{k-1},x_{k+1},\dots,x_n), \\
\partial_n^{1,k}(x_1,\dots, x_n) & = & (x_1\utr x_k,\dots, x_{k-1}\utr x_k,x_{x+1}\otr x_k,\dots,x_n\otr x_k) \\
\end{eqnarray*}
and for $f:C_n(X,R)\to R$ we have
\[\delta^n(f)=f\partial_{n+1}\]
define a chain complex $(C_n,\partial_n)$ and cochain complex $(C^n,\delta^n)$
whose homology and cohomology groups $H_n^{BR}(X;R)$ and $H^n_{BR}(X;R)$ are the 
\textit{birack homology and cohomology} of $X$. The subcomplex 
$(C^D_n(X;R),\partial_n)$ generated by elements of the form $(x_1,\dots,x_n)$ 
where $x_i=x_{i+1}$ for some $i=1,\dots, n-1$ and its dual
are the \textit{degenerate subcomplexes}, modding out by which
yields the \textit{biquandle chain complex} $(C^B_n(X;R),\partial_n)$
and \textit{biquandle cochain complex} $(C_B^n(X;R),\delta^n)$ yielding the
\textit{biquandle homology and cohomology} $H^B_n(X;R)$ and $H_B^n(X;R)$
of $X$.

For quasi-trivial biquandles we can generalize this degenerate subcomplex 
further:
say a chain is \textit{quasi-trivial} if it is an $R$-linear combinations of 
tuples of the form $(x_1,\dots x_n)$ where all $x_j$ lie in the same component
sub-biquandle of $X$. The quasi-trivial chains $C^{DR}_n(X;R)$ clearly form a 
subcomplex
of the birack complex of $X$; indeed, it is exactly the direct sum of the 
birack complexes of the orbit sub-biquandles considered separately as 
biquandles in their own right. Hence we
can form the quotient complex $(C^{QT}_n(X;R)=C^{B}_n(X;R)/C^{DR}_n(X;R)$
and its dual subcomplex to obtain the \textit{quasi-trivial biquandle 
homology and cohomology} of $X$, $H^{QT}_n(X;R)$ and $H_{QT}^n(X;R)$. Such 
quasi-trivial cocycles enable us to enhance the link-homotopy biquandle 
counting invariant.

\begin{df}
Let $X$ be a quasi-trivial biquandle, $R$ a commutative ring and 
$\phi\in C^2_{QT}(X;R)$ a quasi-trivial biquandle 2-cocycle. Then for an oriented
link diagram $D$ representing a link $L$, we define a multiset 
$\Phi_X^{M,\phi}(L)$ as follows:
\begin{itemize}
\item For each $X$-coloring $D_f$ of $D$, we sum the contributions from 
each crossing as depicted to obtain the \textit{Boltzmann weight} 
$BW(D_f)\in R$
\[\includegraphics{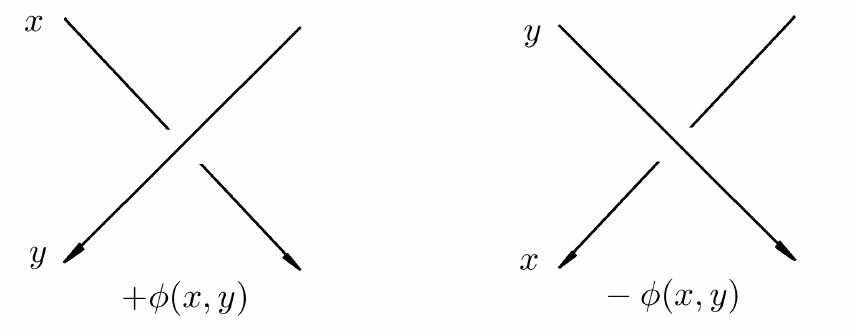}\]
\item We collect these Boltzmann weights over the set of all $X$-colorings of
$D$ to obtain the multiset 
\[\Phi_X^{M,\phi}(L)=\{BW(D_f)\ |\ D_f\ X\mathrm{-coloring\ of\ } D\}.\]
\item If $R=\mathbb{Z}$ or $\mathbb{Z}_n$, we can write a ``polynomial'' form
of the multiset:
\[\Phi_X^{\phi}(L)=\sum_{D_f\ X\mathrm{-coloring\ of\ }D} u^{BW(D_f)}.\]
\end{itemize}
\end{df}
 
It is a standard observation (see \cite{EN}) that $\Phi_X^{M,\phi}(L)$ and
$\Phi_X^{\phi}(L)$ are unchanged by Reidemeister moves; we now observe that
the contribution to each Boltzmann weight from single-component crossings
is zero for any $\phi\in C^2_{QT}(X;R)$, so these invariants are unchanged by
self-crossing changes. Hence we obtain

\begin{pro}
Let $X$ be a quasi-trivial biquandle and $\phi\in C^2_{QT}(X;R)$ a quasi-trivial cocycle
with values in a commutative ring $R$. Then if $L$ and $L'$ are link-homotopic,
we have $\Phi_X^{\phi}(L)=\Phi_X^{\phi}(L')$ and 
$\Phi_X^{M,\phi}(L)=\Phi_X^{M,\phi}(L')$.
\end{pro}

\begin{ex}
Inoue used a $12$-element quasi-trivial quandle showing that the Borromean rings $L6a1$ are not trivial under link-homotopy (see figure $5$ on page $8$ in~\cite{I}).  Here we give 
another proof using a smaller quasi-trivial biquandle. Let $X$ be the $4$-element
biquandle in Example \ref{ex:redb}. Then both the Borromean rings and the unlink
of three components U3 have 64 colorings by $X$, but the cocycle enhancement 
with $\phi:X\times X\to\mathbb{Z}_3$ defined by $\phi=\chi_{(3,2)}+\chi_{(4,2)}$ 
detects the difference with $\Phi^{\phi}_X(L6a1)=48u+16\ne64=\Phi^{\phi}_X(U3)$.
\end{ex}

\begin{ex}
Let $X$ be the quasi-trivial biquandle with operation matrix
\[M_X=\left[\begin{array}{rrrrr|rrrrr}
1 & 1 & 1 & 2 & 3 & 1 & 1 & 1 & 1 & 3 \\
2 & 2 & 2 & 3 & 1 & 2 & 2 & 2 & 2 & 1 \\
3 & 3 & 3 & 1 & 2 & 3 & 3 & 3 & 3 & 2 \\
4 & 4 & 4 & 4 & 4 & 4 & 4 & 4 & 4 & 4 \\
5 & 5 & 5 & 5 & 5 & 5 & 5 & 5 & 5 & 5  
\end{array}\right].\]
We selected three quasi-trivial 2-cocycles over $R=\mathbb{Z}_3,$
\begin{eqnarray*}
\phi_1 & = & 2\chi_{2,4}+2\chi_{2,5}+2\chi_{3,4}+2\chi_{4,5}+\chi_{5,2} \\
\phi_2 & = & 2\chi_{2,5}+2\chi_{3,4}+2\chi_{3,5}+2\chi_{4,1}+2\chi_{4,2}+2\chi_{4,5}+2\chi_{5,1}+2\chi_{5,4}\\
\phi_3 & = & \chi_{1,5}+2\chi_{3,4}+\chi_{4,1}+\chi_{4,2}+\chi_{4,3}+2\chi_{4,5}+\chi_{5,1}+2\chi_{5,2}+2\chi_{5,3}\\
\end{eqnarray*}
found by our \texttt{Python}
code and computed the invariant values for a choice of orientation of 
each of the prime links with up to seven crossings as listed at \cite{KA}.
The results are collected in the table.
\[\scalebox{0.8}{$
\begin{array}{|c|ccccc|}\hline
 L& L2a1 & L4a1 & L5a1 & L6a1 & L6a2  \\ \hline
\Phi_X^{\phi_1} & 2u^2+17 & 2u + 17 & 25 & 2u + 17 & 6u+19  \\
\Phi_X^{\phi_2} & 6u^2+2u+11 & 2u^2+6u+11& 25 & 2u^2+6u+11 & 6u^2+19 \\
\Phi_X^{\phi_3} & 8u^2+11 &8u^2+11 & 25 & 8u+11 & 6u^2+19   \\\hline
 L&  L6a3  &L6a4 & L6a5 &L6n1&  L7a1   \\ \hline
\Phi_X^{\phi_1} & 6u+19 & 6u+38 &  6u+65 & 6u+65 & 25\\
\Phi_X^{\phi_2} & 6u^2+19 & 6u^2+9u+29 &6u^2+36u+29& 6u^2+36u+29 &25\\
\Phi_X^{\phi_3} & 6u^2+19 & 15u+29 & 42u+29 & 42u+29 & 25\\\hline
L &  L7a2 & L7a3& L7a4 & L7a5 & L7a6   \\ \hline
\Phi_X^{\phi_1} & 2u+17 & 25 & 25 & 2u+17 & 2u^2+17 \\
\Phi_X^{\phi_2} & 2u^2 + 6u+11 & 25 & 25 & 6u^2+2u+11 &6u^2+2u+11 \\
\Phi_X^{\phi_3} & 8u^2+11 & 25 & 25 & 8u^2+11 & 8u^2+11 \\ \hline
L & & L7a7 & L7n1 &L7n2 & \\ \hline
\Phi_X^{\phi_1} & & 2u+75 & 2u^2+17 & 25 &\\
\Phi_X^{\phi_2} & & 2u^2+12u+63 & 6u^2+2u+11 &25 &\\
\Phi_X^{\phi_3} & & 14u+63 & 8u^2+11 & 25 &\\ \hline
\end{array}$}
\]
\end{ex}

\begin{ex}
As with other biquandle cocycle invariants, quasi-trivial biquandle cocycle
invariants in general are sensitive to mirror image. $\Phi_X^{\phi}$ with
quasi-trivial biquandle $X$ given by
\[M_x=\left[\begin{array}{rrrr|rrrr}
1 & 1 & 2 & 2 & 1 & 1 & 2 & 2 \\
2 & 2 & 1 & 1 & 2 & 2 & 1 & 2 \\
3 & 3 & 3 & 3 & 4 & 4 & 3 & 3 \\
4 & 4 & 4 & 4 & 3 & 3 & 4 & 4
\end{array}\right]\] and 2-cocycle $\phi\in C^2(X;\mathbb{Z}_8)$ given by
\[\phi=3\chi_{1,4}+6\chi_{2,3}+\chi_{2,4}+2\chi_{3,1}+4\chi_{3,2}+2\chi_{4,1}+4\chi_{4,2}\]
distinguishes the $(4,2)$-torus link $L4a1$ from its mirror image 
$\overline{L4a1}$ with $\Phi_X^{\phi}(L4a1)=8u^7+8\ne8u+8=\Phi_X^{\phi}(\overline{L4a1})$. 
\[\includegraphics{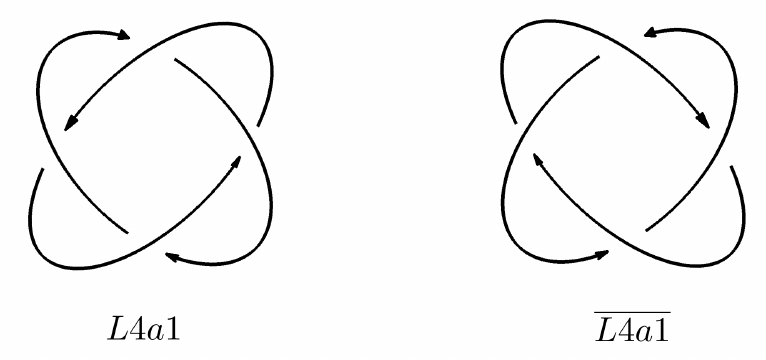}\]
\end{ex}

\section{Open Questions}
In this section we end with some open problems for future research.
\begin{itemize}
	\item
	Habegger and Lin \cite{HL} gave an algorithm determining when two links are equivalent under link-homotopy.  Their algorithm uses an inductive computation of certain cosets in free abelian groups.  Can one give a more easily computable algorithm determining when two pretzel links $(p_1, \ldots, p_m)$ and $(q_1, \ldots, q_n)$ are equivalent under link-homotopy?
	
	\item
	 Can the results on pretzel links in this paper be generalized to Montesinos links?  In particular when is a Montesinos link homotopically trivial?
	
	\item Beyond cocycle enhancements, what other enhancements of the quandle and biquandle counting invariant can detect link homotopy class?
	
\end{itemize}

\subsection*{Acknowledgements}
The authors would like to thank Francis Bonahon for fruitful discussions.
The third listed author was partially supported by Simons Foundation 
collaboration grant 316709.

  \end{document}